\documentclass{amsart}
\usepackage{amsmath, amssymb, amsthm, amsfonts}
\usepackage{hyperref}
\usepackage{cleveref}
\usepackage{tikz-cd}
\usepackage{blkarray}

\usepackage{enumerate}

\newcommand\Sing{\textup{Sing\,}}
\newcommand\codim{\textup{codim\,}}
\newcommand\rank{\textup{rank\,}}
\newcommand\GL{\textup{GL}}

\theoremstyle{plain}
\newtheorem{Thm}{Theorem}[section]
\newtheorem{Prop}[Thm]{Proposition}

\newtheorem{Lem}[Thm]{Lemma}

\theoremstyle{definition}
\newtheorem{Def}[Thm]{Definition}

\newtheorem{Notation}[Thm]{Notation}

\theoremstyle{remark}
\newtheorem{Rmk}[Thm]{Remark}
\newtheorem{Ex}[Thm]{Example}

\crefname{Ex}{Example}{Examples}
\crefname{Prop}{Proposition}{Propositions}

\begin{document}

\title[Determinantal characterization]{Determinantal characterization of higher secant varieties of minimal degree}
\author[Junho Choe and Sijong Kwak]{Junho Choe and Sijong Kwak${}^*$}
\address{Junho Choe \\
School of Mathematics, Korea Institute for Advanced Study (KIAS), 85 Hoegiro Dongdaemun-gu, Seoul 02455, Republic of Korea}
\email{junhochoe@kias.re.kr}
\address{Sijong Kwak\\
Department of Mathematics, Korea Advanced Institute of Science and Technology (KAIST), 373-1 Gusung-dong, Yusung-Gu, Daejeon, Republic of Korea}
\email{sjkwak@kaist.ac.kr}
\thanks{${}^*$ Corresponding author.}

\date{\today}

\begin{abstract}
A variety of minimal degree is one of the basic objects in projective algebraic geometry and has been classified and characterized in many aspects. On the other hand, there are also minimal objects in the category of higher secant varieties, and their algebraic and geometric structures seem to share many similarities with those of varieties of minimal degree.

We prove in this paper that higher secant varieties of minimal degree have determinantal presentation of two types, i.e.\ scroll type and Veronese type. Our result generalizes the del Pezzo-Bertini classification for varieties of minimal degree. Also, as a consequence, we show that for any smooth projective variety having higher secant variety of minimal degree, the embedding line bundle admits a special decomposition into two line bundles as so do those of the well-known examples: varieties of minimal degree, smooth del Pezzo varieties, Segre varieties and $2$-Veronese varieties.
\end{abstract}

\maketitle
\tableofcontents \setcounter{page}{1}

\section{Introduction}

Throughout this paper, we work over an algebraically closed field of characteristic zero, let $X$ be a nondegenerate projective (irreducible and reduced) variety in $\mathbb P^r$, and take $q$ to be a positive integer. The {\it $q$-secant variety} $S^q(X)\subseteq\mathbb P^r$ to $X$ is the Zariski closure of the set of points spanned by $q$ linearly independent points in $X$. From now on, write $e=\codim S^q(X)$ for convenience.

As for projective varieties, it is very interesting to find a degree lower bound for higher secant varieties in terms of the codimension $e$. The following is known due to Ciliberto-Russo: The inequality
$$
\deg S^q(X)\geq\binom{e+q}{q}
$$
holds by \cite[Theorem 4.2]{ciliberto2006varieties} and is sharp for each pair $(e,q)$. They also studied various projective varieties whose degree of the $q$-secant variety attains the minimum. One says that $S^q(X)$ is a {\it $q$-secant variety of minimal degree} if the equality above holds, which generalizes a classical object, namely a {\it variety of minimal degree}. 

Higher secant varieties of minimal degree display many parallels with varieties of minimal degree. They admit natural generalizations of the following results about varieties of minimal degree: (1) a theorem on the number of quadrics by G.\ Castelnuovo \cite{castelnuovo1889ricerche}, (2) the $K_{p,1}$ theorem of M. Green \cite{green1984koszul}, (3) a characterization of $2$-regular projective varieties due to Eisenbud-Goto \cite{eisenbud1984linear}, and (4) the rigidity of property $N_p$ introduced by Eisenbud-Green-Hulek-Popescu \cite{eisenbud2005restricting}.
For details, refer to \cite[Theorem 1.1]{choe2022matryoshka}.

On the other hand, for the classical case $q=1$, the {\it del Pezzo-Bertini classification} (see \cite[Theorem 1]{eisenbud1987varieties}) shows that a variety of minimal degree is either
\begin{enumerate}
\item a rational normal scroll; or

\item a cone over the $2$-Veronese surface $\nu_2(\mathbb P^2)\subset\mathbb P^5$
\end{enumerate} 
as long as it has codimension greater than one. Notice that a rational normal scroll $S(a_1,\ldots,a_n)$ is defined by $2$-minors of a matrix $M$ of linear forms, or simply a {\it linear matrix}, that is {\it $1$-generic}, where
$$
M=
\begin{pmatrix}
x_{1,0}&x_{1,1}&\cdots&x_{1,a_1-1}&\cdots&x_{n,0}&x_{n,1}&\cdots&x_{n,a_n-1}\\
x_{1,1}&x_{1,2}&\cdots&x_{1,a_1}&\cdots&x_{n,1}&x_{n,2}&\cdots&x_{n,a_n}
\end{pmatrix},
$$
and similarly, the generic symmetric linear matrix 
$$
\begin{pmatrix}
x_{0,0}&x_{0,1}&x_{0,2}\\
x_{0,1}&x_{1,1}&x_{1,2}\\
x_{0,2}&x_{1,2}&x_{2,2}
\end{pmatrix}
$$
works for the $2$-Veronese surface. That is, any variety of minimal degree has {\it determinantal presentation} if its codimension is at least $2$.

In this paper, we generalize the del Pezzo-Bertini classification to higher secant varieties of minimal degree in view of determinantal presentation. Let us introduce some terminology. Take $M$ to stand for a $1$-generic linear matrix. If there is a $(q+1)\times(p+q)$ (resp.\ $(q+2)\times(q+2)$ symmetric) linear matrix $M$ whose $2$-minors vanish on $X$, and if $e=p$ (resp.\ $e=3$), then the set of $(q+1)$-minors of $M$ defines $S^q(X)$ ideal-theoretically; in this case, one says that $S^q(X)$ is of {\it scroll type} (resp.\ {\it Veronese type}), and $M$ is a {\it determinantal presentation} of $S^q(X)$. (The scroll type can appear for any positive codimension $e\geq 1$ while any $q$-secant variety of Veronese type must have codimension $e=3$.) Our main result is as follows:
\begin{Thm}\label{detpre}
Suppose that $S^q(X)$ is a $q$-secant variety of minimal degree with codimension $e\geq 2$. Then, $S^q(X)$ is of either
\begin{enumerate}
\item scroll type; or

\item Veronese type with $e=3$.
\end{enumerate}
When $e=3$, the $q$-secant variety $S^q(X)$ is of Veronese type, but not of scroll type if and only if a general $(q-1)$-tangential projection of $X$ is a cone over the $2$-Veronese surface $\nu_2(\mathbb P^2)\subset\mathbb P^5$.
\end{Thm} 

Keep in mind that a general $(q-1)$-tangential projection of $X$ always has minimal degree when $S^q(X)$ has minimal degree \cite[Theorem 4.2(iv)]{ciliberto2006varieties}. Refer to an elementary observation (see \Cref{operate}) that for a general point $z\in X$, if $S^q(X)$ is of scroll type, then 
\begin{enumerate}
\item $S^q(X_z)$ is of scroll type when $e\geq 2$; and 
\item $S^{q-1}(X_{\mathbb T_zX})$ is of scroll type when $q\geq 2$,
\end{enumerate}
and on the other hand, if $S^q(X)$ is of Veronese type, then 
\begin{enumerate}
\item $S^q(X_z)$ is of scroll type; and 
\item $S^{q-1}(X_{\mathbb T_zX})$ is of Veronese type when $q\geq 2$.
\end{enumerate}

Moreover, one can see that these determinantal presentations provide a complete description of the minimal free resolution of $S^q(X)$: The Eagon-Northcott complex \cite{eagon1962ideals} gives the minimal free resolution for higher secant varieties of scroll type, and so does the J{\'o}zefiak complex \cite{jozefiak1978ideals} for higher secant varieties of Veronese type. So, as already observed in \cite[Theorem 1.1(2)]{choe2022matryoshka}, the minimal free resolution of the homogeneous ideal of $S^q(X)$ must be linear of weight $q+1$.

Further, \Cref{detpre} can be delivered in a more intrinsically geometric way by the following theorem:

\begin{Thm}\label{detpresm}
Suppose that $S^q(X)$ has minimal degree of codimension $e\geq 2$, and that $X$ is smooth. Write $L=\mathcal O_X(1)$, and let $|V|\subseteq|L|$ be the embedding linear system of $X\subset\mathbb P^r$. Then, there are two linear systems $|V_i|\subseteq|L_i|$ on $X$ and an effective divisor $B\subset X$ such that
\begin{enumerate}
\item we have a decomposition
$$
L(-B)=L_1\otimes L_2
$$
and a well-defined multiplication map $V_1\otimes V_2\to V$; and

\item either
\begin{enumerate}[ \normalfont (a)]
\item we obtain $\dim|V_1|=q$ and $\dim|V_2|=e+q-1$; or

\item both $|V_i|\subseteq|L_i|$ coincide together with $\dim|V_i|=q+1$ and $e=3$; and
\end{enumerate}

\item each of $|V_i|$ has no fixed components.
\end{enumerate}
\end{Thm}

Here the divisor $B$ has a role as in \Cref{nonempty}, and basic examples of the decomposition are the following:
\begin{itemize}
\item $\mathcal O_S(H)=\mathcal O_S(qF)\otimes\mathcal O_S(H-qF)$ for any rational normal scroll $S=S(a_1,\ldots,a_n)$ with two of the $a_i$ larger than or equal to $q$;

\item $\mathcal O_{\mathbb P^2}(3)=\mathcal O_{\mathbb P^2}(1)\otimes\mathcal O_{\mathbb P^2}(2)$ for the del Pezzo surface $\nu_3(\mathbb P^2)\subset\mathbb P^9$ with $q=2$;

\item $\mathcal O_{\mathbb P^q\times\mathbb P^b}(1,1)=\mathcal O_{\mathbb P^q\times\mathbb P^b}(1,0)\otimes\mathcal O_{\mathbb P^q\times\mathbb P^b}(0,1)$ for the Segre variety $\sigma(\mathbb P^q\times\mathbb P^b)\subset\mathbb P^{(q+1)(b+1)-1}$; and

\item $\mathcal O_{\mathbb P^{q+1}}(2)=\mathcal O_{\mathbb P^{q+1}}(1)\otimes\mathcal O_{\mathbb P^{q+1}}(1)$ for the $2$-Veronese variety $\nu_2(\mathbb P^{q+1})\subset\mathbb P^{\binom{q+3}{2}-1}$.
\end{itemize}
Consult with \Cref{RNS,,dP,,Seg,,2V}. The first three correspond to the case (a), but the last one corresponds to the case (b). \Cref{detpresm} still holds in some sense when $X$ is singular; see \Cref{singular}.

The proof of \Cref{detpre} is performed by an inductive process, and the process is based on the gluing of linear matrices. When constructing a desired linear matrix, we shall consider linear matrices already obtained for general inner and tangential projections of $X$ according to the induction hypothesis, and glue them together by taking account of common block matrices (refer to \Cref{scrollglue,scrollglueII,Veroneseglue}).

This paper is organized as follows: In \Cref{HSV}, we recall the definition of higher secant varieties of minimal degree, and in \Cref{matrix}, we discuss linear matrices, especially gluing procedures of them, in relation with equations of higher secant varieties. Finally, we prove our main theorems in \Cref{proofs}, and put comments on our results in \Cref{comments}.

\bigskip

\section{Preliminaries on higher secant varieties}\label{HSV}

The {\it $q$-secant variety} to $X$ is defined to be
$$
S^q(X)=\overline{\bigcup\langle z_1,\ldots,z_q\rangle}\subseteq\mathbb P^r,
$$
where the union runs over all of the $q$ linearly independent points $z_i$ in $X$. To study $q$-secant varieties for any $q\geq 1$, namely {\it higher secant varieties}, we frequently use {\it linear projections} since they have many interactions with higher secant varieties. For basic notions, facts and notations about linear projections, consult with \cite{choe2022matryoshka}.

As mentioned in the introduction, we are interested in higher secant varieties of minimal degree. 

\begin{Def}[{cf.\ Ciliberto-Russo \cite[Definition 4.4]{ciliberto2006varieties}}]
We say that $S^q(X)$ is a {\it $q$-secant variety of minimal degree} if its degree attains the minimum, that is,
$$
\deg S^q(X)=\binom{e+q}{q}.
$$
\end{Def}

\begin{Ex}\label{fundaEx}
If $X$ is one of the projective varieties below, then $S^q(X)$ has minimal degree:
\begin{enumerate}
\item rational normal scrolls for $q\geq 1$;

\item smooth del Pezzo surfaces for $q\geq 2$;

\item Segre varieties $\sigma(\mathbb P^a\times\mathbb P^b)$ for $q\geq\min\{a,b\}$; and

\item $2$-Veronese varieties $\nu_2(\mathbb P^n)$ for $q\geq n-1$.
\end{enumerate}
For more examples, see \cite{ciliberto2006varieties}.
\end{Ex}

The following are known results on higher secant varieties of minimal degree:

\begin{Prop}[{Ciliberto-Russo \cite[Theorem 4.2]{ciliberto2006varieties}}]
If $S^q(X)$ is a $q$-secant variety of minimal degree, then for a general point $z\in X$,
\begin{enumerate}
\item $S^q(X_z)$ is a $q$-secant variety of minimal degree; and

\item $S^{q-1}(X_{\mathbb T_zX})$ is a $(q-1)$-secant variety of minimal degree.
\end{enumerate}
\end{Prop}

\begin{Rmk}
In fact, its converse holds in the following manner: Let $z\in X$ stand for a general point. Suppose that either
\begin{enumerate}
\item $S^q(X_z)$ has minimal degree with $e\geq2$ and $q\geq 1$; or

\item $S^{q-1}(X_{\mathbb T_zX})$ has minimal degree with $e\geq1$ and $q\geq 2$.
\end{enumerate}
Then, $S^q(X)$ has minimal degree. Refer to \cite[Theorem 1.4]{choe2022matryoshka}.
\end{Rmk}

\bigskip

\section{Linear matrices and higher secant varieties}\label{matrix}

In what follows, we use the following:

\begin{Notation}\label{notation}
\phantom{...}
\begin{itemize}
\item For a linear matrix $M$ (i.e.\ a matrix of linear forms), $I_s(M)$ is the ideal generated by its $s$-minors.

\item The labels added to a matrix indicate the size of submatrices; for example, the submatrix $M$ of an $(a+1)\times(b+1)$ matrix
$$
\begin{blockarray}{ccc}
&1&b\\
\begin{block}{c(cc)}
1&\ast&\ast\\
a&\ast&M\\
\end{block}
\end{blockarray}
$$
has size $a\times b$.

\item For matrices, the entry subscripts start with zero; for example, the entry $M_{0,0}$ of a matrix $M$ lies in the first row and the first column.

\item For convenience, $E^{a,b}$ stands for various matrices over the base field whose $(i,j)$-entry $E^{a,b}_{i,j}$ is zero except $E^{a,b}_{a,b}\neq 0$.

\item We denote by $\GL(a)$ the general linear group of degree $a$ over the base field together with $\GL(a,b)=\GL(a)\times\GL(b)$ for simplicity.

\item We write $S$ for the homogeneous coordinate ring of $\mathbb P^r$, and $I_X$ for the homogeneous ideal of $X\subseteq\mathbb P^r$.

\item For a point $z\in\mathbb P^r$, we mean by $\overline{z}$ its image under a projection map if the projection is well understood.
\end{itemize}
\end{Notation}

A well-known connection exists between linear matrices and equations of higher secant varieties.

\begin{Prop}
Let $M$ be a linear matrix. If $\rank M(z)\leq\rho$ for every $z\in X$, then $\rank M(w)\leq q\cdot\rho$ for every $w\in S^q(X)$. In other words, $I_s(M)\subseteq I_X$ implies $I_{q(s-1)+1}(M)\subseteq I_{S^q(X)}$.
\end{Prop}

This fact is used to construct significant equations of higher secant varieties to certain projective varieties. Especially, in tensor geometry, {\it flattenings} are used in such a way for Segre varieties, Veronese varieties, etc. For this purpose, one shall use linear matrices that satisfy a genericity property.

\begin{Def}[D. Eisenbud {\cite[Proposition-Definition 1.1]{eisenbud1988linear}}]
A linear matrix $M$ of size $a\times b$ is {\it $1$-generic} if $AMB^T$ has no zero entries for any $(A,B)\in\GL(a,b)$.
\end{Def}

We collect important structural results on linear matrices as follows:

\begin{Thm}\label{structure}
Let $M$ be an $a\times b$ linear matrix with $a\leq b$.
\begin{enumerate}
\item \textup{(Eagon-Northcott \cite{eagon1962ideals})} We have $\codim I_a(M)\leq b-a+1$, and if the equality holds, then $S/I_a(M)$ is minimally resolved by the Eagon-Nortcott complex.

\item \textup{(D.\ Eisenbud \cite{eisenbud1988linear})} If $M$ is $1$-generic, then $S/I_a(M)$ is a Cohen-Macaulay domain of codimension $b-a+1$. 

\item \textup{(T.\ J{\'o}zefiak \cite{jozefiak1978ideals})} When $M$ is symmetric, we have $\codim I_{a-1}(M)\leq 3$, and if the equality holds, then $S/I_{a-1}(M)$ is minimally resolved by the J{\'o}zefiak complex \cite[p.\ 600]{jozefiak1978ideals}.

\item If $M$ is symmetric and $1$-generic, then $S/I_{a-1}(M)$ is a Cohen-Macaulay ring of codimension $3$.
\end{enumerate}
\end{Thm}

\begin{proof} For lack of references, we prove (4) here. Note that the determinant of $M$ is an irreducible polynomial by (2). Let $Y$ be the induced hypersurface, and take a general smooth point $z\in Y$ so that $M(z)$ has rank $a-1$. Replacing $M$ with a suitable product $AMA^T$, $A\in\GL(a)$, if necessary, we may assume that
$$
M(z)=
\begin{blockarray}{ccc}
&a-1&1\\
\begin{block}{c(cc)}
a-1&I&0\\
1&0&0\\
\end{block}
\end{blockarray}
\ .
$$
Set $M'$ (resp.\ $M''$) to be the linear matrix obtained by deleting the first (resp.\ last) row of $M$. Then, $M'$ and $M''$ are $1$-generic of size $(a-1)\times a$, hence $I_{a-1}(M')$ and $I_{a-1}(M'')$ are prime ideals of codimension $2$ by (2). Note that $I_{a-1}(M')\neq I_{a-1}(M'')$ because $\rank M'(z)\neq\rank M''(z)$. So, one has
$$
\codim I_{a-1}(M)\geq\codim(I_{a-1}(M')+I_{a-1}(M''))>\codim I_{a-1}(M')=2.
$$ 
By (3), we are done.
\end{proof}

We now define certain linear matrices that play an important role.

\begin{Def}
Let $M$ be a $1$-generic linear matrix whose $2$-minors vanish on $X$, and $p$ a positive integer. Then, $M$ is
\begin{enumerate}
\item a {\it $(p,q)$-scroll matrix} of $X$ if it has size $(q+1)\times(p+q)$; and

\item a {\it $q$-Veronese matrix} of $X$ if it is symmetric with $q+2$ rows.
\end{enumerate}
\end{Def}

\bigskip

\subsection{Gluing scroll and Veronese matrices}

In this subsection, we present ways to construct scroll and Veronese matrices from smaller ones.

\begin{Def}
Let $M$ be a linear matrix of rank at most one on $X$, and take $\Gamma$ to be a possibly empty finite subset of $X$. Then, $M$ is {\it $\Gamma$-regular} if $\bigcap_{z\in\Gamma}\ker M(z)$ has codimension $|\Gamma|$, where $\ker$ means the right null space.
\end{Def}

Let $p$ be a positive integer, and $\Gamma\subset X$ a finite subset. Notice that if $M$ is a $\Gamma$-regular $(p,q)$-scroll matrix, then so is $AMB^T$ for every $(A,B)\in\GL(q+1,p+q)$. The same thing holds for any $q$-Veronese matrix $M$ and the $\GL(q+2)$-action $A\mapsto AMA^T$.

\begin{Def}
Let $\Gamma$ be a finite subset of $X$. Then, we define
\begin{align*}
\Sigma_{p,q}(X,\Gamma) & =\frac{\{\text{$\Gamma$-regular $(p,q)$-scroll matrices of $X$}\}}{\GL(q+1,p+q)} \text{ and} \\
V_q(X,\Gamma) & =\frac{\{\text{$\Gamma$-regular $q$-Veronese matrices of $X$}\}}{\GL(q+2)}.
\end{align*}
Additionally, a subset 
$$
\Sigma'_{p,q}(X,\Gamma)\subset\Sigma_{p,q}(X,\Gamma)\
$$ 
is obtained by collecting $\Gamma$-regular linear matrices whose transposes are also $\Gamma$-regular.
\end{Def}

\begin{Rmk}\label{operate}
Let $M$ be a linear matrix of size $a\times b$ that satisfies $M(z)=E^{0,0}$ for a point $z\in\mathbb P^r$ (see \Cref{notation}). Consider the following operations:
\begin{enumerate}
\item $\pi_zM$ is obtained from $M$ by deleting the first column;

\item $\pi_z^TM$ is the transpose of $\pi_zM$; and

\item $\partial_zM$ is obtained from $M$ by deleting both the first row and column. 
\end{enumerate}
One may see that for any finite points $z\in\Gamma\subset X$, the operations above give the maps below:
$$
\begin{tikzcd}
&\Sigma'_{p,q-1}(X_{\mathbb T_zX},\Gamma_{\mathbb T_zX})\\
\Sigma'_{p-1,q}(X_z,\Gamma_z) \ar[d, hook] &\Sigma'_{p,q}(X,\Gamma) \ar[l,"\pi_z"] \ar[u,"\partial_z",swap] \ar[d, hook]  \\
\Sigma_{p-1,q}(X_z,\Gamma_z)&\Sigma_{p,q}(X,\Gamma) \ar[l,"\pi_z"]
\end{tikzcd}
\begin{tikzcd}
&V_{q-1}(X_{\mathbb T_zX},\Gamma_{\mathbb T_zX})\\
\Sigma'_{2,q}(X_z,\Gamma_z)&V_q(X,\Gamma) \ar[l,"\pi_z^T"] \ar[u,"\partial_z",swap]
\end{tikzcd}
$$
Here we mean by $\Gamma_\Lambda$ the projection $\pi_\Lambda(\Gamma\setminus\Lambda)$, and two sets 
$$
\Sigma'_{p,0}(X_{\mathbb T_zX},\Gamma_{\mathbb T_zX})\quad\text{and}\quad V_0(X_{\mathbb T_zX},\Gamma_{\mathbb T_zX})
$$ 
have the natural definitions:
\begin{itemize}
\item If $[M]\in\Sigma'_{p,0}(X_{\mathbb T_zX},\Gamma_{\mathbb T_zX})$, then $M$ is a row vector of $p$ linearly independent linear forms vanishing on $\mathbb T_zX$ such that both $M$ and $M^T$ are $\Gamma_{\mathbb T_zX}$-regular.

\item If $[M]\in V_0(X_{\mathbb T_zX},\Gamma_{\mathbb T_zX})$, then $M$ is a $1$-generic and symmetric $2\times 2$ matrix of linear forms vanishing on $\mathbb T_zX$ such that it has determinant in $I_{X_{\mathbb T_zX}}$ and is $\Gamma_{\mathbb T_zX}$-regular.
\end{itemize}
Similarly, so does a set $V_{-1}(X_\Lambda,\Gamma_\Lambda)$ when $\Lambda=\langle\mathbb T_zX,\mathbb T_wX\rangle$ for two (general) points $z,w\in X$:
\begin{itemize}
\item If $[M]\in V_{-1}(X_\Lambda,\Gamma_\Lambda)$, then $M$ is a $\Gamma_\Lambda$-regular $1\times 1$ matrix whose entry is a nonzero linear form vanishing on $\Lambda$.
\end{itemize}
\end{Rmk}

The following lemma is useful to verify whether two given linear matrices are equivalent:

\begin{Lem}\label{LinCombi}
Let $M$ be an $a\times b$ linear matrix whose $2$-minors vanish on $X$ with $a,b\geq 2$. Suppose that the submatrix $M'$ of
$$
M=
\begin{blockarray}{ccc}
&1&b-1\\
\begin{block}{c(cc)}
a & A & M' \\
\end{block}
\end{blockarray}
$$
is $1$-generic. If $M$ is not $1$-generic, then $A$ is a linear combination of columns of $M'$ with coefficients in the base field.
\end{Lem}

\begin{proof}
Let $\Bbbk$ be the base field, and write $\Bbbk^c$ for the $\Bbbk$-vector space of column vectors with $c$ entries. Since $M$ is not $1$-generic, but since so is $M'$, there are two nonzero elements $v_0\in\Bbbk^a$ and $w_0\in\Bbbk^b$ such that $v_0^TMw_0=0$ holds, and the first entry of $w_0$ is nonzero. We will show that 
$$
Mw_0=0.
$$ 
Set $v_1\in\Bbbk^a$ to be arbitrary, and choose a nonzero vector $w_1\in\Bbbk^b$ whose first entry is zero. Consider a determinant of the form
$$
\begin{vmatrix}
v_0^TMw_0 & v_0^TMw_1 \\
v_1^TMw_0 & v_1^TMw_1
\end{vmatrix}
.
$$
By assumption, it lies in $I_X$ with $v_0^TMw_0=0$ and $v_0^TMw_1\neq0$, hence $v_1^TMw_0=0$. Since $v_1$ was arbitrary in $\Bbbk^a$, we are done.
\end{proof}

Now, we formalize our methods of gluing scroll and Veronese matrices as follows:

\begin{Prop}[Gluing for scroll matrices I]\label{scrollglue}
Let $p\geq 3$ be an integer, and $\Gamma$ a general finite subset of $X$ with two different points $z,w\in\Gamma$. Then, $\Sigma_{p,q}(X,\Gamma)$ is naturally identified with the subset 
$$
\Sigma=\{([M_z],[M_w])\}\subset\Sigma_{p-1,q}(X_z,\Gamma_z)\times\Sigma_{p-1,q}(X_w,\Gamma_w)
$$ 
given by 
\begin{enumerate}
\item[$(\pi\pi)$] $\pi_{\overline{w}}[M_z]=\pi_{\overline{z}}[M_w]$.
\end{enumerate}
\end{Prop}

\begin{proof}
Let $\Sigma$ be the set of $([M_z],[M_w])$ in question. Then, two operations $\pi_z$ and $\pi_w$ induce a map $\Sigma_{p,q}(X,\Gamma)\to\Sigma$ by 
$$
[M]\mapsto(\pi_z[M],\pi_w[M]).
$$ 
We will construct its inverse map as follows: Pick any element $([M_z],[M_w])\in\Sigma$. Then by $(\pi\pi)$, we may assume that 
\begin{align*}
M_z&=
\begin{blockarray}{ccc}
&1&p+q-2\\
\begin{block}{c(cc)}
q+1&B & C\\
\end{block}
\end{blockarray}
\quad\text{with}\quad B(w)\neq 0\text{; and}\\
M_w&=
\begin{blockarray}{ccc}
&1&p+q-2\\
\begin{block}{c(cc)}
q+1&A & C\\
\end{block}
\end{blockarray}
\quad\text{with}\quad A(z)\neq 0,
\end{align*}
where $[C]$ lies in $\Sigma_{p-2,q}(X_{\langle z,w\rangle},\Gamma_{\langle z,w\rangle})$. Put
$$
M=
\begin{pmatrix}
A&B&C
\end{pmatrix}
.
$$
Since $\rank M_z(y),\rank M_w(y)\leq 1$ and $C(y)\neq 0$ for a general point $y\in X$, we have $\rank M(y)\leq 1$ on $X$. To show that $M$ is $1$-generic, suppose not. Then by \Cref{LinCombi}, $A$ is a linear combination of $B$ and columns of $C$. However, it is a contradiction to $A(z)\neq 0$; $M$ is thus $1$-generic. And it is easy to see that $M$ is $\Gamma$-regular, and so the mapping $([M_z],[M_w])\mapsto[M]$ is the desired inverse map.
\end{proof}

\begin{Prop}[Gluing for scroll matrices II]\label{scrollglueII}
Set $q\geq 2$, and take $\Gamma$ to be a general finite subset of $X$ with two distinct points $z,w\in\Gamma$. Then, under the assumption that $\Sigma'_{2,q}(X_{\langle z,w\rangle},\Gamma_{\langle z,w\rangle})=\emptyset$, there is a map to $\Sigma_{3,q}(X,\Gamma)$ from the subset 
$$
\Sigma'=\{([M_z],[M_w],[M_{\mathbb T_zX}],[M_{\mathbb T_wX}])\}
$$ 
of $\Sigma'_{2,q}(X_z,\Gamma_z)\times\Sigma'_{2,q}(X_w,\Gamma_w)\times\Sigma'_{3,q-1}(X_{\mathbb T_zX},\Gamma_{\mathbb T_zX})\times\Sigma'_{3,q-1}(X_{\mathbb T_wX},\Gamma_{\mathbb T_wX})$ determined by 
\begin{enumerate}
\item[$(\partial\partial)$] $\partial_{\overline{w}}[M_{\mathbb T_zX}]=\partial_{\overline{z}}[M_{\mathbb T_wX}]$;

\item[$(\partial\pi)$] $\partial_{\overline{w}}[M_z]=\pi_{\overline{z}}[M_{\mathbb T_wX}]$ and  $\partial_{\overline{z}}[M_w]=\pi_{\overline{w}}[M_{\mathbb T_zX}]$.
\end{enumerate}
\end{Prop}

\begin{proof}
Choose an arbitrary element $([M_z],[M_w],[M_{\mathbb T_zX}],[M_{\mathbb T_wX}])\in\Sigma'$. By the gluing for scroll matrices I (\Cref{scrollglue}), it is enough to verify 
\begin{enumerate}
\item[$(\pi\pi)$] $\pi_{\overline{w}}[M_z]=\pi_{\overline{z}}[M_w]$.
\end{enumerate} 
By $(\partial\partial)$, we may write $M_{\mathbb T_zX}$ and $M_{\mathbb T_wX}$ as
\begin{align*}
M_{\mathbb T_zX}&=
\begin{blockarray}{ccc}
&1&q+1\\
\begin{block}{c(cc)}
1&\ast & F_0\\
q-1&\ast & I\\
\end{block}
\end{blockarray}
\quad\text{with}\quad M_{\mathbb T_zX}(w)=E^{0,0}\text{; and}\\
M_{\mathbb T_wX}&=
\begin{blockarray}{ccc}
&1&q+1\\
\begin{block}{c(cc)}
1&\ast & C\\
q-1&\ast & I\\
\end{block}
\end{blockarray}
\quad\text{with}\quad M_{\mathbb T_wX}(z)=E^{0,0},
\end{align*}
where $[I]\in\Sigma'_{3,q-2}(X_\Lambda,\Gamma_\Lambda)$ for $\Lambda=\langle\mathbb T_zX,\mathbb T_wX\rangle$. Also, due to $(\partial\pi)$, we are able to transform $M_z$ into 
$$
M_z=
\begin{blockarray}{ccc}
&1&q+1\\
\begin{block}{c(cc)}
1&b&C\\
1&e&F\\
q-1&H&I\\
\end{block}
\end{blockarray}
\quad\text{with}\quad M_z(w)=E^{1,0},
$$
where $e$ stands for a linear form. Similarly, one gets
$$
M_w=
\begin{blockarray}{ccc}
&1&q+1\\
\begin{block}{c(cc)}
1&\ast& C_0\\
1&\ast& F_0\\
q-1&G& I\\
\end{block}
\end{blockarray}
\quad\text{with}\quad M_w(z)=E^{0,0}.
$$
Apply both \Cref{LinCombi} and the assumption to
$$
\begin{pmatrix}
C_0^T&C^T&F^T&I^T
\end{pmatrix}
\quad\text{and}\quad
\begin{pmatrix}
F_0^T&C^T&F^T&I^T
\end{pmatrix}
$$
in order to yield 
$$
M_w=
\begin{blockarray}{ccc}
&1&q+1\\
\begin{block}{c(cc)}
1&\ast&C\\
1&\ast&F\\
q-1&G&I\\
\end{block}
\end{blockarray}
\quad\text{with}\quad M_w(z)\in\langle E^{0,0},E^{1,0}\rangle
$$
up to the natural action, hence $(\pi\pi)$ holds.
\end{proof}

On the other hand, in order to glue for Veronese matrices, we first establish a map associated to general finite points $\Gamma\subset X$ with two distinct points $z,w\in\Gamma$. Let $\Sigma'\subset V_{q-1}(X_{\mathbb T_zX},\Gamma_{\mathbb T_zX})\times\Sigma'_{2,q}(X_w,\Gamma_w)$ be the fiber product
$$
\begin{tikzcd}
\Sigma' \ar[r] \ar[d] & \Sigma'_{2,q}(X_w,\Gamma_w) \ar[d,"\partial_{\overline z}"] \\
V_{q-1}(X_{\mathbb T_zX},\Gamma_{\mathbb T_zX}) \ar[r,"\pi_{\overline w}^T", swap] & \Sigma'_{2,q-1}(X_{\langle\mathbb T_zX,w\rangle},\Gamma_{\langle\mathbb T_zX,w\rangle}).
\end{tikzcd}
$$
Choose a pair $([M_{\mathbb T_zX}],[M_w])$ contained in $\Sigma$. By its nature, we may assume that 
\begin{align*}
M_{\mathbb T_zX}&=
\begin{blockarray}{ccc}
&1&q\\
\begin{block}{c(cc)}
1&d&E\\
q&E^T&F\\
\end{block}
\end{blockarray}
\quad\text{with}\quad M_{\mathbb T_zX}(w)=E^{0,0}\text{; and}\\
M_w&=
\begin{blockarray}{cccc}
&1&1&q\\
\begin{block}{c(ccc)}
1&a&b&C\\
q&C_0^T&E^T&F\\
\end{block}
\end{blockarray}
\quad\text{with}\quad M_w(z)=E^{0,0}.
\end{align*}
Then, one easily shows that the linear matrix
$$
M_z=
\begin{pmatrix}
b&d&E\\
C^T&E^T&F
\end{pmatrix}
$$
is a $\Gamma_z$-regular $(2,q)$-scroll matrix of $X_z$ with $M_z^T$ being $\Gamma_z$-regular as well, hence $[M_z]\in\Sigma'_{2,q}(X_z,\Gamma_z)$.

\begin{Def}
Let $\Gamma,z,w,\Sigma'$ be as above. Then, the map described above is written as 
$$
D^z_w:\Sigma'\to\Sigma'_{2,q}(X_z,\Gamma_z).
$$
\end{Def}

Veronese matrices have their own gluing process as follows:

\begin{Prop}[Gluing for Veronese matrices]\label{Veroneseglue}
Let $\Gamma\subset X$ be general finite points of $X$ with $z,w\in\Gamma$ two different points. Then, the set $V_q(X,\Gamma)$ is naturally identified with the subset 
$$
V=\{([M_z],[M_w],[M_{\mathbb T_zX}],[M_{\mathbb T_wX}])\}
$$ 
of $\Sigma'_{2,q}(X_z,\Gamma_z)\times\Sigma'_{2,q}(X_w,\Gamma_w)\times V_{q-1}(X_{\mathbb T_zX},\Gamma_{\mathbb T_zX})\times V_{q-1}(X_{\mathbb T_wX},\Gamma_{\mathbb T_wX})$
given by
\begin{enumerate}
\item[$(\partial\partial)$] $\partial_{\overline{w}}[M_{\mathbb T_zX}]=\partial_{\overline{z}}[M_{\mathbb T_wX}]$;

\item[$(\partial\pi^T)$] $\partial_{\overline{w}}[M_z]=\pi_{\overline{z}}^T[M_{\mathbb T_wX}]$, $\partial_{\overline{z}}[M_w]=\pi_{\overline{w}}^T[M_{\mathbb T_zX}]$;

\item[$(D)$] $D^z_w([M_{\mathbb T_zX}],[M_w])=[M_z]$ and $D^w_z([M_{\mathbb T_wX}],[M_z])=[M_w]$.
\end{enumerate}
\end{Prop}

\begin{proof}
Sending
$$
[M]\mapsto(\pi_z^T[M],\pi_w^T[M],\partial_z[M],\partial_w[M]),
$$
we have a map $V_q(X,\Gamma)\to V$. Now, choose any element $([M_z],[M_w],[M_{\mathbb T_zX}],[M_{\mathbb T_wX}])\in V$. By $(\partial\partial)$, we may assume that
\begin{align*}
M_{\mathbb T_zX}&=
\begin{blockarray}{ccc}
&1&q\\
\begin{block}{c(cc)}
1&d&E\\
q&E^T&F\\
\end{block}
\end{blockarray}
\quad\text{with}\quad M_{\mathbb T_zX}(w)=E^{0,0},\text{ and}\\
M_{\mathbb T_wX}&=
\begin{blockarray}{ccc}
&1&q\\
\begin{block}{c(cc)}
1&a&C\\
q&C^T&F\\
\end{block}
\end{blockarray}
\quad\text{with}\quad M_{\mathbb T_wX}(z)=E^{0,0}
\end{align*}
for some element $[F]\in V_{q-2}(X_\Lambda,\Gamma_\Lambda)$ with $\Lambda=\langle\mathbb T_zX,\mathbb T_wX\rangle$. Use $(\partial\pi^T)$ so that we have
$$
M_w=
\begin{blockarray}{cccc}
&1&1&q\\
\begin{block}{c(ccc)}
1&\ast&b_0&C_0\\
q&\ast&E^T&F\\
\end{block}
\end{blockarray}
\quad\text{with}\quad M_w(z)=E^{0,0}
$$
up to the canonical action. Then, $(D)$ applies, hence we get a representative
$$
M_z=
\begin{pmatrix}
b_0 &d&E\\
C_0^T&E^T&F
\end{pmatrix}
\quad\text{with}\quad M_z(w)=E^{0,1}.
$$
However, $(\partial\pi^T)$ says that
$$
\left[
\begin{pmatrix}
C_0^T&F
\end{pmatrix}
\right]=\left[
\begin{pmatrix}
C^T&F
\end{pmatrix}
\right].
$$
Note that the linear matrix
$$
\begin{pmatrix}
C_0^T&C^T&F
\end{pmatrix}
$$
is not $1$-generic; if it were, then its $q$-minors would be linearly independent (see \Cref{structure}(1) and (2)), a contradiction to $\left[
\begin{pmatrix}
C_0^T&F
\end{pmatrix}
\right]=\left[
\begin{pmatrix}
C^T&F
\end{pmatrix}
\right]$. Therefore, by \Cref{LinCombi}, $C_0$ is a linear combination of $C$ and columns of $F$, and so one can transform $M_z$ into
$$
M_z=
\begin{pmatrix}
b &d&E\\
C^T&E^T&F
\end{pmatrix}
\quad\text{with}\quad M_z(w)=E^{0,1}.
$$
Also, applying $D^w_z$ to $([M_{\mathbb T_wX}],[M_z])$, we have 
$$
M_w=
\begin{blockarray}{cccc}
&1&1&q\\
\begin{block}{c(ccc)}
1&a&b&C\\
q&C^T&E^T&F\\
\end{block}
\end{blockarray}
\quad\text{with}\quad M_w(z)=E^{0,0}
$$
up to the natural action. Then, the linear matrix
$$
M=
\begin{pmatrix}
a&b&C\\
b&d&E\\
C^T&E^T&F
\end{pmatrix}
$$
is evidently a $\Gamma$-regular $q$-Veronese matrix of $X$ and satisfies 
$$
\pi_z^T[M]=[M_z],\quad\pi_w^T[M]=[M_w],\quad\partial_z[M]=[M_{\mathbb T_zX}]\quad\text{and}\quad\partial_w[M]=[M_{\mathbb T_wX}].
$$
In conclusion, we have constructed the inverse of the map $V_q(X,\Gamma)\to V$ above.
\end{proof}

\bigskip

\subsection{Determinantal presentations}

In this subsection, we give the definition and examples of determinantal presentations. The examples consist of those of higher secant varieties to the projective varieties in \Cref{fundaEx} and show us what happens in our main theorems.

\begin{Def}\label{scroll}
One says $S^q(X)$ to be of
\begin{enumerate}
\item {\it scroll type} if $X$ has an $(e,q)$-scroll matrix $M$; and

\item {\it Veronese type} if $X$ has a $q$-Veronese matrix $M$ with $e=3$. 
\end{enumerate}
In both cases, $M$ is called a {\it determinantal presentation} of $S^q(X)$.
\end{Def}

\begin{Ex}\label{RNS}
Let us consider any rational normal scroll $X=S(a_1,\ldots,a_n)$. Then it is defined by the vector bundle $E=\bigoplus_{i=1}^n\mathcal O(a_i)$ on $\mathbb P^1$ and its set $H^0(\mathbb P^1,E)=\bigoplus_iH^0(\mathbb P^1,\mathcal O(a_i))$ of global sections. More precisely, when $H$ is a hyperplane section of $X$ and $F$ is a ruling $(n-1)$-plane in $X$, if $\{s,t\}$ is a basis for $H^0(X,\mathcal O_X(F))$, and if $\lambda_i$ is the generator of the direct summand in $H^0(X,\mathcal O_X(H-a_iF))=\bigoplus_jH^0(\mathbb P^1,\mathcal O(a_j-a_i))$ with respect to the index $j=i$, then the global sections 
$$
x_{i,j}=\lambda_is^{a_i-j}t^j\in H^0(X,\mathcal O_X(H))
$$ 
give the standard embedding of $X$. Let us take into account the linear matrix
$$
M=
\begin{blockarray}{cccccccc}
\otimes&\lambda_1s^{a_1-q}&\cdots&\lambda_1t^{a_1-q}&\cdots&\lambda_ns^{a_n-q}&\cdots&\lambda_nt^{a_n-q}\\ \cline{1-8}
\begin{block}{c(ccccccc)}
s^q&x_{1,0}&\cdots&x_{1,a_1-q}&&x_{n,0}&\cdots&x_{n,a_n-q}\\
\vdots&\vdots&&\vdots&\cdots&\vdots&&\vdots\\
t^q&x_{1,q}&\cdots&x_{1,a_1}&&x_{n,q}&\cdots&x_{n,a_n}\\
\end{block}
\end{blockarray}
$$
determined by the multiplication map 
$$
H^0(X,\mathcal O_X(qF))\otimes H^0(X,\mathcal O_X(H-qF))\to H^0(X,\mathcal O_X(H)).
$$
By its nature, $M$ is $1$-generic and has $2$-minors in $I_X$. It is well known that $e+q$ is equal to the number of columns of $M$, hence $M$ is an $(e,q)$-scroll matrix of $X$. Consequently, higher secant varieties to any rational normal scroll is of scroll type.

But the $q$-secant variety to the rational normal curve in $\mathbb P^{2q+2}$ is of Veronese type at the same time; the determinantal presentation is the linear matrix
$$
M=
\begin{blockarray}{cccc}
\otimes &s^{q+1}&\cdots&t^{q+1}\\ \cline{1-4}
\begin{block}{c(ccc)}
s^{q+1}&x_{0}&\cdots&x_{q+1}\\
\vdots&\vdots&&\vdots\\
t^{q+1}&x_{q+1}&\cdots&x_{2q+2}\\
\end{block}
\end{blockarray}
$$
that comes from the mulitiplication map 
$$
H^0(\mathbb P^1,\mathcal O(q+1))\otimes H^0(\mathbb P^1,\mathcal O(q+1))\to H^0(\mathbb P^1,\mathcal O(2q+2)).
$$
\end{Ex}

Now, the rest of linear matrices in this subsection are $1$-generic for the same reason and have rank less than $2$ on projective varieties to be discussed.

\begin{Ex}\label{dP}
The $2$-secant variety to every smooth del Pezzo surface is of scroll type except the embedding of $\mathbb P^1\times\mathbb P^1$ by the linear system $|\mathcal O(2,2)|$; in this case, the $2$-secant variety is of Veronese type.

First, we consider the $3$-Veronese surface $X=\nu_3(\mathbb P^2)\subset\mathbb P^9$ for simplicity. Fix a basis $\langle t_0,t_1,t_2\rangle=H^0(\mathbb P^2,\mathcal O_{\mathbb P^2}(1))$. Note that $H^0(\mathbb P^2,\mathcal O_{\mathbb P^2}(3))$ is generated by $$
x_{[i,j,k]}=t_it_jt_k
$$ 
for all multisets $[i,j,k]$, $0\leq i,j,k\leq 2$. Then the $x_{[i,j,k]}$ form homogeneous coordinates of $\mathbb P^9$. In this coordinate system, the linear matrix
$$
M=
\begin{blockarray}{ccccc}
\otimes&t_0^2&t_0t_1&\cdots&t_2^2\\ \cline{1-5}
\begin{block}{c(cccc)}
t_0&x_{[0,0,0]}&x_{[0,0,1]}&\cdots&x_{[0,2,2]}\\
t_1&x_{[1,0,0]}&x_{[1,0,1]}&\cdots&x_{[1,2,2]}\\
t_2&x_{[2,0,0]}&x_{[2,0,1]}&\cdots&x_{[2,2,2]}\\
\end{block}
\end{blockarray}
$$
is a $(4,2)$-scroll matrix of $X$. By a theorem of Severi \cite{severi1901intorno}, $S^2(X)$ has codimension $4$, hence it is of scroll type with determinantal presentation $M$.

Second, for the blowup $X\subset\mathbb P^{9-s}$ of $s$ points $\Gamma\subset\mathbb P^2$ in general position, its $2$-secant variety is of scroll type due to the linear matrix associated to the multiplication map 
$$
H^0(\mathbb P^2,\mathcal O(1))\otimes H^0(\mathbb P^2,\mathcal I_\Gamma(2))\to H^0(\mathbb P^2,\mathcal I_\Gamma(3)),
$$
where $\mathcal I_\Gamma$ is the ideal sheaf of $\Gamma\subset\mathbb P^2$.

Last, let $X$ be the complete embedding of $\mathbb P^1\times\mathbb P^1$ by the line bundle $\mathcal O(2,2)$. Set a basis $\{s,t\}$ (resp.\ $\{u,v\}$) for the vector space $H^0(\mathbb P^1,\mathcal O(1))$ of linear forms on the first (resp.\ second) factor of $\mathbb P^1\times\mathbb P^1$. So, $X$ is determined by the global sections
$$
x_{i,j}=s^{2-i}t^iu^{2-j}v^j\in H^0(\mathbb P^1\times\mathbb P^1,\mathcal O(2,2)),\quad0\leq i,j\leq 2.
$$
Then, the multiplication map
$$
H^0(\mathbb P^1\times\mathbb P^1,\mathcal O(1,1))\otimes H^0(\mathbb P^1\times\mathbb P^1,\mathcal O(1,1))\to H^0(\mathbb P^1\times\mathbb P^1,\mathcal O(2,2))
$$
corresponds to the linear matrix
$$
\begin{blockarray}{ccccc}
\otimes&su&sv&tu&tv\\ \cline{1-5}
\begin{block}{c(cccc)}
su&x_{0,0}&x_{0,1}&x_{1,0}&x_{1,1}\\
sv&x_{0,1}&x_{0,2}&x_{1,1}&x_{1,2}\\
tu&x_{1,0}&x_{1,1}&x_{2,0}&x_{2,1}\\
tv&x_{1,1}&x_{1,2}&x_{2,1}&x_{2,2}\\
\end{block}
\end{blockarray}
\ .
$$
Having it as the determinantal presentation, $S^2(X)$ is of Veronese type.
\end{Ex}

\begin{Ex}\label{Seg}
The Segre variety $X=\sigma(\mathbb P^a\times\mathbb P^b)$ with $a\leq b$ has $a$-secant variety of scroll type. Let $s_0,\ldots,s_a$ (resp. $t_0,\ldots,t_b$) be a basis for $H^0(\mathbb P^a,\mathcal O(1))$ (resp. $H^0(\mathbb P^b,\mathcal O(1))$). Then, the global sections $$
x_{i,j}=s_it_j\in H^0(\mathbb P^a\times\mathbb P^b,\mathcal O(1,1))
$$ 
induce the embedding of $X$. As easily seen, $S^a(X)$ is of scroll type, and its determinantal presentation is
$$
M=
\begin{blockarray}{ccccc}
\otimes&t_0&t_1&\cdots&t_b\\ \cline{1-5}
\begin{block}{c(cccc)}
s_0&x_{0,0}&x_{0,1}&\cdots&x_{0,b}\\
\vdots&\vdots&\vdots&&\vdots\\
s_a&x_{a,0}&x_{a,1}&\cdots&x_{a,b}\\
\end{block}
\end{blockarray}
\ .
$$

It can be said that the $a$-secant variety $S^a(X)$ stands for the {\it generic} case of scroll type since $M$ is a generic matrix. In particular, for any projective variety $X$, if $S^q(X)$ is of scroll type, then $X$ lies in a cone over a linear section of $\sigma(\mathbb P^q\times\mathbb P^{e+q-1})$.
\end{Ex}

\begin{Ex}\label{2V}
For the $2$-Veronese variety $X=\nu_2(\mathbb P^n)$, its $(n-1)$-secant variety is of Veronese type, but its $n$-secant variety is of scroll type. Let $t_0,\ldots,t_n$ be homogeneous coordinates of $\mathbb P^n$ so that the homogeneous $2$-forms
$$
x_{[i,j]}=t_it_j
$$ 
define the embedding of $X$. Consider the following symmetric linear matrix:
$$
M=
\begin{blockarray}{cccc}
\otimes&t_0&\cdots&t_n\\ \cline{1-4}
\begin{block}{c(ccc)}
t_0&x_{[0,0]}&\cdots&x_{[0,n]}\\
\vdots&\vdots&&\vdots\\
t_n&x_{[n,0]}&\cdots&x_{[n,n]}\\
\end{block}
\end{blockarray}
$$
It is given by the multiplication map
$$
H^0(\mathbb P^n,\mathcal O(1))\otimes H^0(\mathbb P^n,\mathcal O(1))\to H^0(\mathbb P^n,\mathcal O(2)).
$$
Notice that Terracini's lemma yields 
$$
\codim S^{n-1}(X)=3\quad\text{and}\quad\codim S^n(X)=1.
$$ 
By \Cref{structure}(4), $S^{n-1}(X)$ is of Veronese type with determinantal presentation $M$. However, the $n$-secant variety $S^n(X)$ is of scroll type, for the determinant of $M$ defines it. 

As above, $S^{n-1}(X)$ is the {\it generic} case of Veronese type because $M$ is a generic symmetric matrix, which implies that for any projective variety $X$, if $S^q(X)$ is of Veronese type, then $X$ is contained in a cone over a linear section of $\nu_2(\mathbb P^{q+1})$.

For more information, when $q$ is in the range $1\leq q\leq n-2$, the $q$-secant variety to $X$ is of neither scroll type nor Veronese type. Indeed, with assistance of Terracini's lemma, a general $(q-1)$-tangential projection of $X$ is projectively equivalent to the $2$-Veronese variety $\nu_2(\mathbb P^{n-q+1})$ that is not a variety of minimal degree. By \cite[Theorem 4.2(iv)]{ciliberto2006varieties}, $S^q(X)$ does not have minimal degree.
\end{Ex}

\bigskip

\section{Proofs of the main theorems}\label{proofs}

First of all, we prove the uniqueness of determinantal presentations which helps us to glue scroll and Veronese matrices.

\begin{Prop}[Uniqueness]\label{uniqueness}
If $S^q(X)$ is of either
\begin{enumerate}
\item scroll type with $e\geq 2$; or

\item Veronese type, 
\end{enumerate}
then the determinantal presentation is unique up to the natural actions. 
\end{Prop}

\begin{proof}
It is enough to show that the following always holds:
\begin{enumerate}
\item $|\Sigma_{e,q}(X,\Gamma)|\leq 1$ for $e+q-1$ general points $\Gamma\subset X$ when $e\geq 2$; and

\item $|V_q(X,\Gamma)|\leq 1$ for $q+2$ general points $\Gamma\subset X$ when $e=3$.
\end{enumerate}
For if there were two distinct determinantal presentations $M$ and $M'$ of size $a\times b$, then taking $b-1$ or $b$ general points $\Gamma\subset X$, one would find that they are both $\Gamma$-regular, hence $[M]\neq[M']$ in either $\Sigma_{e,q}(X,\Gamma)$ or $V_q(X,\Gamma)$.

By \Cref{scrollglue,Veroneseglue}, we need only consider
\begin{enumerate}
\item $\Sigma_{2,q}(X,\Gamma)$ in setting of $e=2$ and $q\geq 1$; and

\item $V_1(X,\Gamma)$ in setting of $e=3$ and $q=1$.
\end{enumerate}
If $[M]$ lies in $\Sigma_{2,q}(X,\Gamma)$ when $e=2$, then by (1) and (2) of \Cref{structure}, it is uniquely determined by the minimal free resolution of $I_{S^q(X)}$. If $[M]$ lies in $V_1(X,\Gamma)$ when $e=3$, then its uniqueness holds for the following reason: Indeed, in order to use the gluing for Veronese matrices (\Cref{Veroneseglue}) again, suppose given an element $[N]\in V_0(X_{\mathbb T_zX},\Gamma_{\mathbb T_zX})$ for any point $z\in\Gamma$. Writing $\Gamma=\{z,w,y\}$, we may assume that
$$
N=
\begin{pmatrix}
a&b\\
b&c
\end{pmatrix}
\quad\text{with}\quad N(w)=E^{0,0}\text{ and }N(y)=E^{1,1}.
$$
Refer to the fact that $X_{\mathbb T_zX}$ has codimension at most one in this case. Then $\langle\mathbb T_zX,\mathbb T_wX\rangle$, $\langle\mathbb T_zX,\mathbb T_yX\rangle$ and $X_{\mathbb T_zX}$ are defined by $c$, $a$ and $ac-b^2$, respectively. Thus, $[N]$ is a unique element of $V_0(X_{\mathbb T_zX},\Gamma_{\mathbb T_zX})$.
\end{proof}

The following lemma is necessary when using the Hilbert-Burch theorem, but quite tricky:

\begin{Lem}\label{wild}
Let $M$ be an $a\times b$ linear matrix with $a\geq q+1$ and $b\geq q+2$. Assume that $I_{q+1}(M)\subseteq I_{S^q(X)}$ holds, and $AMB^T$ has no zero $(q+1)$-minors for any $(A,B)\in\GL(a,b)$. Then, $M$ is $1$-generic with $I_2(M)\subseteq I_X$.
\end{Lem}

\begin{proof}
In this proof, we will use the Koszul cohomology and a Koszul class; for their definitions, we cite \cite{green1984koszul} and \cite{aprodu2010koszul}. If $M$ were not $1$-generic, then we would be able to construct a nonzero Koszul class $\gamma\in K_{1,q+1}(I_{S^q(X)},V)$ such that the rank of the induced map $\gamma:V^\ast\to(I_{S^q(X)})_{q+1}$ is less than $q+2$, a contradiction by \cite[Corollary 4.3]{choe2022matryoshka}. Indeed, we may assume that the size of $M$ is $(q+1)\times(q+2)$, and to the contrary, suppose that for some integer $i_0\geq 1$, the first row $(l_0,l_1,\ldots,l_{q+1})$ of $M$ satisfies 
$$
\text{$l_i=0$ for all $0\leq i<i_0$, but $l_{i_0},\ldots,l_{q+1}$ are linearly independent.}
$$ 
Then by assumption, the Koszul class
$$
\gamma=\sum_{i=i_0}^{q+1}(-1)^il_i\otimes\Delta_i
$$
is the desired one, where $\Delta_i$ is the $(q+1)$-minor of $M$ obtained by deleting the column $M_{\ast,j}$. Thus, $M$ is $1$-generic.

Note that minors of $M$ having various sizes are all nonzero. Now, if $I_2(M)\not\subseteq I_X$ held, then for a general point $z\in X$, one would find a hypersurface of degree at most $q-1$ containing $S^{q-1}(X_{\mathbb T_zX})$, a contradiction. Therefore, $I_2(M)$ is contained in $I_X$.
\end{proof}

We are ready to prove our main theorems.

\begin{proof}[\bf Proof of \Cref{detpre}]
We will prove a stronger statement: For a general $(q-1)$-tangential projection $X_\Lambda$ of $X$,
\begin{enumerate}
\item[(S)] if $X_\Lambda$ is a rational normal scroll together with $e\geq 2$, then 
$$
|\Sigma_{e,q}(X,\Gamma)|=1
$$
holds for $e+q-1$ general points $\Gamma\subset X$; and

\item[(V)] if $X_\Lambda$ is a cone over the $2$-Veronese surface, then 
$$
|V_q(X,\Gamma)|=1
$$
holds for $q+2$ general points $\Gamma\subset X$. 
\end{enumerate}
Note that if there is an $(e,q)$-scroll matrix of $X$, then $X_\Lambda$ becomes a rational normal scroll, and that by the uniqueness (\Cref{uniqueness}), it is enough to show the existence of an element in either $\Sigma_{e,q}(X,\Gamma)$ or $V_q(X,\Gamma)$.

(S) We use double induction on $e\geq 2$ and $q\geq 1$. Refer to the fact that when $X_\Lambda$ is a rational normal scroll, so is a general $(q-1)$-tangential projection of $X_z$ for a general point $z\in X$.

First, consider the case where either $e=2$ or $q=1$. Suppose that $e=2$. By \cite[Theorem 1.1]{choe2022matryoshka}, $I_{S^q(X)}$ has minimal free resolution
$$
\begin{tikzcd}
S^{q+2}(-q-1) & S^{q+1}(-q-2) \ar[l,"M^T",swap] & 0 \ar[l].
\end{tikzcd}
$$
The Hilbert-Burch theorem and \Cref{wild} tell us that $M$ is a $(2,q)$-scroll matrix of $X$, and after taking $q+1$ general points $\Gamma\subset X$, the linear matrix $M$ becomes $\Gamma$-regular. Assume that $q=1$. Since $X$ is a rational normal scroll, it has an $(e,1)$-scroll matrix $M$. For $e$ general points $\Gamma\subset X$, the set $\Sigma_{e,1}(X,\Gamma)$ contains $[M]$ as an element.

Second, we see the case where $e=3$ and $q\geq 2$. Proceed to apply the gluing for scroll matrices II (\Cref{scrollglueII}) so that it shall suffice to show the following: For two general points $z,w\in X$, each of two sets
$$
\Sigma'_{3,q-2}(X_{\langle\mathbb T_zX,\mathbb T_wX\rangle},\emptyset)\quad\text{and}\quad\Sigma'_{2,q-1}(X_{\langle\mathbb T_zX,w\rangle},\emptyset)
$$
has at most one element. It is true by the uniqueness (\Cref{uniqueness}) together with the definition of $\Sigma'_{3,0}(X_{\langle\mathbb T_zX,\mathbb T_wX\rangle},\emptyset)$ for the subcase $q=2$. Consequently, we get an (unique) element $[M]\in\Sigma_{3,q}(X,\{z,w\})$, and it lies in $\Sigma_{3,q}(X,\Gamma)$ after choosing $q+2$ general points $\Gamma\subset X$.

Last, for the case $e\geq 4$, we may use the gluing for scroll matrices I (\Cref{scrollglue}), hence $\Sigma_{e,q}(X,\Gamma)$ is not empty for $e+q-1$ general points $\Gamma\subset X$. Therefore, we have finished the proof of (S).

(V) Note that for a general inner projection $X_z$, its general $(q-1)$-tangential projection is a rational normal scroll. By induction on $q\geq 1$ and the part (S), the gluing for Veronese matrices (\Cref{Veroneseglue}) works the same way. We are done.
\end{proof}

A proof of \Cref{detpresm} is given as follows:

\begin{proof}[\bf Proof of \Cref{detpresm}]
It is straightforward by \Cref{M2D} below; let $L$ and $M$ be the line bundle $\mathcal O_X(1)$ and the determinantal presentation induced by \Cref{detpre}, respectively.
\end{proof}

It is well known that on any projective variety $X$, a decomposition $L=L_1\otimes L_2$ of a line bundle $L$ into two line bundles $L_i$ yields a matrix $M$ with entries in $H^0(X,L)$ by using the multiplication map
$$
H^0(X,L_1)\otimes H^0(X,L_2)\to H^0(X,L).
$$
After choosing bases $\langle s_0,\ldots,s_{a-1}\rangle=H^0(X,L_1)$ and $\langle t_0,\ldots,t_{b-1}\rangle=H^0(X,L_2)$, $M$ can be explicitly defined by 
$$
M_{i,j}=s_it_j.
$$ 
Then in the case $L=\mathcal O_X(1)$, it is $1$-generic and satisfies that $2$-minors of $M$ are all zero on $X$. The following is a sort of reverse process:

\begin{Lem}\label{M2D}
Assume that $X$ is smooth, and let $M$ be a $1$-generic $a\times b$ linear matrix with $I_2(M)\subseteq I_X$. Then, for the line bundle $L=\mathcal O_X(1)$, the projective variety $X$ admits two linear systems $|V_i|\subseteq|L_i|$ and an effective divisor $B\subset X$ such that
\begin{enumerate}
\item we have
$$
L(-B)=L_1\otimes L_2\quad\text{and}\quad M_{i,j}=s_it_ju
$$
for some bases $\langle s_0,\ldots,s_{a-1}\rangle=V_1$, $\langle t_0,\ldots,t_{b-1}\rangle=V_2$ and a defining section $u$ of $B$;

\item if $M$ is symmetric, then 
$$
L_1=L_2\quad\text{and}\quad s_i=\alpha t_i
$$ 
for some common scalar $\alpha\neq 0$; and

\item each of $|V_i|$ has no fixed components.
\end{enumerate}
\end{Lem}

\begin{proof}
By the property of $M$, we have a well-defined rational map $\phi:X\dashrightarrow\mathbb P^{a-1}$ such that 
$$
\phi(z)=(M_{0,j}(z):\cdots:M_{a-1,j}(z))
$$ 
for any $0\leq j\leq b-1$. Since $X$ is smooth, there exists a vector
$$
(s_0,\ldots,s_{a-1})
$$
for some global sections $s_i$ of a line bundle $L_1$ on $X$ such that it gives
$$
\phi=(s_0:\cdots:s_{a-1})\quad\text{with}\quad\codim\text{Bs}(s_0,\ldots,s_q)>1
$$ 
and is unique up to a scalar multiple (see \cite[p.\ 492]{griffiths1978principles}), where $\text{Bs}$ means the base locus. One finds that the $s_i$ are linearly independent because $M$ is $1$-generic. Take 
$$
V_1=\langle s_0,\ldots,s_{a-1}\rangle\subseteq H^0(X,L_1).
$$
Then since $\codim\text{Bs}|V_1|>1$, for each $j$, the rational section 
$$
\frac{M_{0,j}}{s_0}=\cdots=\frac{M_{a-1,j}}{s_{a-1}}
$$ 
of $L\otimes L_1^{-1}$ becomes a global section $t^0_j\in H^0(X,L\otimes L_1^{-1})$. Take $B$ to be the fixed component of $\langle t^0_0,\ldots,t^0_{b-1}\rangle$, and let $u\in H^0(X,\mathcal O_X(B))$ be its defining section so that we obtain global sections $t_j=t^0_j/u\in H^0(X,L(-B)\otimes L_1^{-1})$. Consequently, we have
$$
M_{i,j}=s_it_ju.
$$
Put
$$
L_2=L(-B)\otimes L_1^{-1}\quad\text{and}\quad V_2=\langle t_0,\ldots, t_{b-1}\rangle\subseteq H^0(X,L_2).
$$ 
Note that $\codim\text{Bs}|V_2|>1$ holds, and the $t_j$ are linearly independent.

Now, suppose that $M$ is symmetric. As the rational map $\phi$ is equal to another rational map 
$$
\psi=(M_{i,0}:\cdots:M_{i,a-1})=(t_0:\cdots:t_{a-1})
$$ 
for any $0\leq i\leq a-1$, there is a nonzero scalar $\alpha$ such that $s_i=\alpha t_i$ for every $i$.
\end{proof}

\bigskip

\section{Comments}\label{comments}

\begin{Rmk}\label{nonempty}
In \Cref{detpresm}, the effective divisor $B$ is possibly nonempty as follows: 
\begin{enumerate}
\item Consider a rational normal scroll $X=S(a_1,\ldots,a_n)$ with $n\geq 2$ and $a_i\geq 1$, and suppose that 
$$
a_1,\ldots,a_{n-1}<q,\quad\text{but}\quad a_n\geq 2q+1.
$$ 
Let us use the notations in \Cref{RNS}. Note that $S^q(X)$ is scroll type with determinantal presentation
$$
\begin{pmatrix}
x_{n,0}&\cdots&x_{n,a_n-q}\\
\vdots&&\vdots\\
x_{n,q}&\cdots&x_{n,a_n}
\end{pmatrix}
.
$$
Since $x_{n,j}=\lambda_ns^{a_n-j}t^j$ on $X$, one has
$$
B=\{\lambda_n=0\}\in|H-a_nF|,
$$ 
and the decomposition in \Cref{detpresm} corresponds to $H-B\sim qF+(a_n-q)F$.

\item Let $X\subset\mathbb P^7$ be a Roth surface (see \cite[Definition 3.1]{ilic1998geometric}) with associated rational normal scroll $S(0,0,5)$ and the vertex line $\ell$ of $S(0,0,5)$. By definition, $X$ is a smooth surface together with $\ell\subset X\subset S(0,0,5)$. We claim that $S^2(X)$ is a $2$-secant variety of minimal degree with codimension $2$, and
$$
B=\ell.
$$
A theorem of Severi \cite{severi1901intorno} says that $\dim S^2(X)=5$. Therefore, $S^2(X)$ and $S^2(S(0,0,5))$ are the same, and thus $S^2(X)$ has minimal degree of codimension $2$. Notice that $\ell$ is cut out by entries of the determinantal presentation of $S^2(S(0,0,5))$.
\end{enumerate}
\end{Rmk}

\begin{Rmk}\label{singular}
One may extend \Cref{detpresm} for the case $X$ singular by considering a resolution $f:Y\to X$ of singularities (or more generally, any surjective morphism from a smooth complete variety); for instance, putting $L=f^\ast\mathcal O_X(1)$, one would get the same decomposition
$$
L(-B)=L_1\otimes L_2
$$
on $Y$.
\end{Rmk}

\begin{Rmk}
In \Cref{detpresm}, if the linear system $|V|$ is complete, then so are both linear systems $|V_i|$. Assume that $|V|=|L|$. Then using the multiplication map 
$$
H^0(X,L_1)\otimes H^0(X,L_2)\to H^0(X,L)
$$ 
followed by the natural map $H^0(X,L(-B))\to H^0(X,L)$, we have a $1$-generic linear matrix $\widetilde{M}$ on $\mathbb P^r$ of size $a\times b$ whose $2$-minors vanish on $X$, where $a=h^0(X,L_1)$ and $b=h^0(X,L_2)$. 

We first consider the case $\dim|V_1|=q$ (i.e.\ scroll type). If $a>q+1$ held, then for any $(q+2)\times(e+q)$ submatrix $\widetilde{N}$ of $\widetilde{M}$, we would have 
$$
\codim S^{q+1}(X)\geq\codim I_{q+2}(\widetilde{N})=\codim S^q(X)-1,
$$ 
a contradiction by \cite[Proposition 1.2.2]{russo2016geometry}. If $b>e+q$ held, then for any $(q+1)\times(e+q+1)$ submatrix $\widetilde{N}$ of $\widetilde{M}$, we would have 
$$
\codim S^q(X)\geq\codim I_{q+1}(\widetilde{N})=\codim S^q(X)+1,
$$ 
a contradiction. Next, let us see the case $\dim|V_1|=q+1$ (i.e.\ Veronese type). If $a>q+2$ held, then for any $(q+3)\times(q+3)$ principal submatrix $\widetilde{N}$ of $\widetilde{M}$, we would have
$$
\codim S^{q+1}(X)\geq\codim I_{q+2}(\widetilde{N})=3=\codim S^q(X).
$$ 
We have shown that $|V_i|=|L_i|$ for all $i=0,1$ in any event.
\end{Rmk}

\begin{Rmk}
Not every minimal degree $q$-secant variety of codimension $1$ is defined by the determinant of a $(q+1)\times(q+1)$ linear matrix. For example, consider the $3$-Veronese surface $X=\nu_3(\mathbb P^2)$ with $q=3$. Then, $S^3(X)$ is cut out by a quartic equation $f$, namely the {\it Aronhold invariant} of degree $4$. To the contrary, suppose that a $4\times 4$ linear matrix $M$ has determinant $f$. 

We claim that $M$ is a $(1,3)$-scroll matrix of $X$. Indeed, it is well known that the Jacobian ideal $J_f$ of $f$ is contained in $I_3(M)$, and that $\Sing S^3(X)=S^2(X)$ in this case. We have $I_{S^2(X)}=\sqrt{J_f}\subseteq\sqrt{I_3(M)}$ and 
$$
4=\codim I_{S^2(X)}\leq\codim I_3(M),
$$
which means that $\codim I_3(M)=4$ by \cite[Theorem 3]{eagon1962ideals}. Hence, the prime ideal $I_{S^2(X)}$ contains $I_3(M)$. Refer to the fact that for the $4\times 4$ generic linear matrix $\widetilde{M}$, the ideal $I_3(\widetilde{M})$ is a Cohen-Macaulay ideal of codimension $4$ by \cite[Corollary 4]{hochster1971cohen}, and $3$-minors of $\widetilde{M}$ are linearly independent. By the Cohen-Macaulayness, no $3$-minors of $AMB^T$ vanish for any $(A,B)\in\GL(4,4)$ since $AMB^T$ is a linear section of $\widetilde{M}$. Apply \Cref{wild} to conclude that $M$ is a $1$-generic linear matrix satisfying $I_2(M)\subseteq I_X$. 

Thus, by definition, $S^3(X)$ is of scroll type with determinantal presentation $M$. Let $L_i$ be the line bundles on $X$ given by \Cref{M2D} so that $\dim|L_i|\geq 3$. Hence, $L_i=\mathcal O_{\mathbb P^2}(d_i)$ for some integers $d_i\geq 2$, a contradiction to $\mathcal O_X(1)=\mathcal O_{\mathbb P^2}(3)$.
\end{Rmk}

\begin{Rmk}
Suppose that $X$ is a smooth curve $C$ of genus $g$. Let $\delta$ be a positive integer, $L$ any very ample line bundle on $C$ of degree $d>\delta+2g-2$, and $p=d-\delta-g-q+1$. Consider $C$ to be a projective curve embedded by the complete linear system $|L|$. Then, we are able to construct a map
$$
\{\text{linear systems }g^q_\delta\text{ on }C\}\to\Sigma_{p,q}(C,\emptyset)
$$
as follows: Take $|V_1|\subseteq|L_1|$ to be an arbitrary linear system $g^q_d$ on $C$. Write $L_2=L\otimes L_1^{-1}$ and $V_2=H^0(C,L_2)$. Since $\deg L_2>2g-2$, we have $\dim V_2=p+q$ by the Riemann-Roch theorem. As usual, using the multiplication map $V_1\otimes V_2\to H^0(C,L)$, we obtain an element $[M]\in\Sigma_{p,q}(C,\emptyset)$. 

Note that by \Cref{M2D}, this map admits a section whose image consists of base point free linear systems, and that the image $[M]$ above lies in $\Sigma_{p,q}(C,\Gamma)$ for a finite subset $\Gamma\subset C$ if $\Gamma$ is separated by $L_2$, that is, $h^0(C,L_2(-\Gamma))=h^0(C,L_2)-|\Gamma|$.

On the other hand, one could conclude a similar thing for $q$-Veronese matrices of $C$.
\end{Rmk}

\bigskip
 
\noindent{\bf Acknowledgments.}
The authors are supported by the National Research Foundation of Korea (NRF) grant funded by the Korea government (2019R1A2C3010487 for the first author and 2021R1A2C1013851 for the second author).
 
\nocite{*}
\bibliographystyle{amsplain}
\bibliography{Det_pre}
\end{document}